\documentclass[leqno]{amsart}
\usepackage{amscd,amsfonts,amssymb}
\usepackage{graphicx}
\usepackage{color}
\newcommand{\e}{\epsilon}

\numberwithin{equation}{section}
\newtheorem{Theorem}{Theorem}[section]
\newtheorem{Proposition}{Proposition}[section]
\newtheorem{Lemma}{Lemma}[section]
\newtheorem{Assumption}{Assumption}
\newtheorem{Corollary}{Corollary}[section]
\theoremstyle{definition}

\theoremstyle{remark}
\newtheorem{Remark}{Remark}[section]

\DeclareMathOperator{\im}{Im}

\newcommand{\D}{\mathrm {d}}

\DeclareMathOperator{\tr}{Tr}
\DeclareMathOperator{\re}{Re}

{\catcode `\@=11 \global\let\AddToReset=\@addtoreset}
\AddToReset{equation}{section}

\newcommand{\N}{{\mathbb N }}
\newcommand{\I}{\mathrm{i}}
\newcommand{\R}{{\mathbb R}}
\newcommand{\E}{\mathrm{e}}

\newcommand{\ie}{{\sl i.e.\/ }}
\newcommand{\cf}{{\sl cf.\/ }}
\newcommand{\eg}{{\sl e.g.\/}}

\newcommand{\be}{\begin{equation}}
\newcommand{\ee}{\end{equation}}

\def\v{{\bf v}}

\def\({\left(}
\def\){\right)}
\def\<{\left\langle}
\def\>{\right\rangle}
\def\O{\mathcal O}

\newcommand{\Id}[1]{{\rm I\kern-2pt I_{#1}}}
\renewcommand{\hbar}{{\displaystyle\bar{\phantom{x}}\kern-6pt h}}

\title[Rotating Superfluids]
{Rigorous derivation of the hydrodynamical equations for rotating superfluids}

\author{Hailiang Liu}
\address{Department of Mathematics, Iowa State University, Ames, IA 50011-2064}
\email{hliu@iastate.edu}

\author[C. Sparber]{Christof Sparber}
\address{Wolfgang Pauli Institute Vienna, Nordbergstra\ss e 15, A-1090 Vienna, Austria and
 DAMTP, Centre for Mathematical Sciences, Wilberforce Road, Cambridge CB3 0WA,
UK.}
\email{christof.sparber@univie.ac.at}

\thanks{C.~S. has been supported by the APART grant of the Austrian Academy of Sciences.
H.~L. was partially supported by the National Science
Foundation under Grant DMS05-05975.}

\keywords{Semi-classical asympotics, nonlinear Schr\"odinger equation, Bose-Einstein condensates, rotational superfluids,
Thomas-Fermi limit}
\begin{document}

\begin{abstract}
Using a modified WKB approach, we present a rigorous semi-classical analysis for solutions
of nonlinear Schr\"{o}dinger equations with rotational forcing.
This yields a rigorous justification for the
hydrodynamical system of rotating superfluids.
In particular it is shown that global-in-time semi-classical convergence
holds whenever the limiting hydrodynamical system has global smooth solutions
and we also discuss the semi-classical dynamics of several physical quantities describing
rotating superfluids.
\end{abstract}

\maketitle

%%%%%%%%%%%%%%%%%%%%%%%%%%%%%%%%%%%%%%%%%%%%%%%%%%%%%%%%%%%%%%%%%%%%%%%%%%%%%%

\section{Physical Motivation}

\emph{Bose-Einstein condensates} (BECs)
play an outstanding role in present-day physics, \cf \cite{PiSt} for a general introduction.
Understanding and controlling their behavior is of great fundamental importance and essential for novel applications.
A particular focus of interest is on dynamical phenomena related to
\emph{superfluidity} and the creation of \emph{quantized vortices}, see, \eg,
\cite{Af, CD, RBD:2002, Do}. To this end the typical experimental
set-up is based on a trapping potential subject to a \emph{rotational
forcing}.
The mathematical description of this system is then usually
given by the celebrated \emph{Gross-Pitaevskii equation} (GPE), a
mean field model for the macroscopic wave function of the
condensate (see \cite{ESY} for a rigorous derivation).
In the rotating reference frame we are consequently led to
\begin{equation}\label{gpe}
\I \hbar \, \partial_t \psi = -\frac{\hbar^2}{2m} \, \Delta \psi
+V_0(x) \psi + \kappa |\psi|^2 \psi +\I  \Omega_0 x^\bot \cdot
\nabla_x\psi,
\end{equation}
with a nonlinear coupling constant $\kappa=N 4\pi \hbar a/m$, where $N$ is the number of
particles forming the condensate and $m$, $a\in \R$ respectively denote their corresponding mass and scattering
length. In the context of BECs the potential $V$ is usually assumed to be a harmonic oscillator confinement, \ie
$V_0(x)= \frac{1}{2}m \omega^2_0 | x|^2$, for some
$\omega_0 \in \R^d$. In \eqref{gpe} we also write $x^\bot= (x_2, -x_1, 0)^\top$ in $d= 3$ spatial dimensions and
analogously $x^\bot= (x_2, -x_1)^\top$ for $d= 2$. The latter case is motivated by recent experiments for
BECs which are strongly confined in one or
two directions.
Thus $x^\bot \cdot
\nabla_x$ can be interpreted as
the negative $x_3$-component of the quantum mechanical \emph{angular momentum operator}
$L= x\times (-\I \nabla_x)$ and $\Omega_0$ is the corresponding angular velocity.

In the so-called \emph{Thomas-Fermi limit}, which typically applies for large systems,
the dynamics is presumably well described by the
\emph{hydrodynamical equations for rotating superfluids} \cite{AfDu, RZS, ZS}, \ie
\begin{equation}\label{super0}
\left \{
\begin{aligned}
& \partial_t \rho +\nabla_x \cdot \big(\rho (\v  -\Omega_0 x^\bot)\big)=0,\\
& \partial_t \v + \nabla_x \left( \frac{|\v|^2}{2}  -\Omega_0 x^\bot \cdot \v + V_0 + \rho \right)=0,
\end{aligned}
\right.
\end{equation}
where $\rho:=|\psi|^2$ denotes the particle density and $\v$ the corresponding superfluid velocity defined by
$$
\v:= \frac{\hbar}{m } \frac{\im \left(\overline \psi \, \nabla_x \psi \right)}{|\psi|^2}.
$$
The passage from \eqref{gpe} to \eqref{super0} is usually explained by using the classical
\emph{Madelung transformation} of the wave function, where one writes
\be \label{mad}
\psi (t,x)= \sqrt{\rho(t,x)} \, \exp\left(\I \Phi(t,x)/\hbar \right),
\ee
and consequently identifies $\v:= \nabla_x \Phi$. Formally plugging the ansatz \eqref{mad} into \eqref{gpe},
separating real and imaginary parts,
and discarding terms $\propto \hbar^2$ yields \eqref{super0}. This asymptotic regime
is particularly interesting for numerical simulations, \cf \cite{BDZ:2006, BaWa}, and we also
refer to \cite{GLM, Car2} for an extensive review of such dispersive limits.
However, the representation \eqref{mad} makes no sense at vacuum points, where $\rho =0$.
Indeed the system \eqref{super0} degenerates at such points
and is therefore only weakly hyperbolic (see, \eg, \cite{GLM}).

In the present work we shall rigorously prove that \eqref{super} approximates \eqref{gpe} in a certain sense and moreover
draw some further conclusions from that. To this end we rescale equation \eqref{gpe}, as described in \cite{BJM, Liu:2006}, into its
dimensionless form. This yields
\begin{equation}\label{sgpe}
\I \epsilon \, \partial_t \psi^\epsilon = -\frac{\epsilon^2}{2} \, \Delta \psi^\epsilon
+V(x) \psi^\epsilon + \e \delta^{5/2}|\psi^\epsilon|^2 \psi^\epsilon +\I \epsilon \, \Omega x^\bot \cdot
\nabla_x\psi^\epsilon,
\end{equation}
with $\e= \hbar/(m L^2)$ and $\delta = 4\pi a N/a_0$, where $L$ denotes the characteristic length of the condensate and
$a_0= \sqrt{\hbar/(m\omega_0)}$ is the ground state length of the harmonic oscillator potential $V(x)$.
From now on $\Omega$ and $V(x)$ denote (rescaled) dimensionless quantities.
The particular choice $|\e \delta^{5/2}|=\O(1)$, which yields
$a_0 \ll L$ and thus $\e =(a_0/4\pi |a|N)^{1/5} \ll 1$, corresponds to the Thomas-Fermi regime of
strong-interactions, \cf \cite{BJM} for more details. Note however that our scaling is different
from the one used in \cite{Af, AfDu, CRDY}.

We consequently study the following \emph{semi-classically scaled} nonlinear Schr\"odinger equation (NLS)
with rotational forcing
\begin{equation}\label{se}
\left \{
\begin{aligned}
\I \epsilon \, \partial_t \psi^\epsilon = & \ -\frac{\epsilon^2}{2} \, \Delta \psi^\epsilon
+V(x) \psi^\epsilon +  f(|\psi^\epsilon|^2) \psi^\epsilon +\I \epsilon \, \Omega x^\bot \cdot
\nabla_x\psi^\epsilon\\
\psi ^\e \big|_{t=0} = & \ \psi^\e_{\rm in}(x),
\end{aligned}
\right.
\end{equation}
where $t \in \R$, $x \in \R^d$, for $d=2,3$, and $\Omega \geq 0$, some given constant (independent of $\e \ll 1$).
The energy functional associated to \eqref{se} reads
\be
\label{energy}
E(\psi^\epsilon)=\int_{\R^d}
\frac{\epsilon^2}{2} \, |\nabla_x \psi^\epsilon|^2
+ \left(V(x) + F(|\psi^\epsilon|^2)\right) |\psi^\epsilon|^2 + \I \epsilon \, \Omega \overline {\psi^\epsilon}
x^\bot \cdot \nabla_x \psi^\e  \ \D x,
\ee
where $F$ is the primitive of $f$. From now
on we impose the following
assumptions.
\begin{Assumption} \label{ass} It holds:
\begin{itemize}
\item The nonlinearity satisfies $f \in C^\infty(\R)$ such that $f'>0$.
\item The potential $V$ is of the form $V(x)= \frac{1}{2} |\omega \cdot x|^2$, with $\omega \in \R^{d}$.
\end{itemize}
\end{Assumption}
The superfluid equations corresponding to the NLS \eqref{se} are thus given by
\begin{equation}\label{super}
\left \{
\begin{aligned}
& \partial_t \rho +\nabla_x \cdot \big(\rho (\v  -\Omega x^\bot)\big)=0,\\
& \partial_t \v + \nabla_x \left( \frac{|\v|^2}{2}  -\Omega x^\bot \cdot \v + V + f(\rho) \right)=0.
\end{aligned}
\right.
\end{equation}

We expect the system \eqref{super} to be a valid description of solutions to \eqref{se} in the limit $\e \to 0$.
In order to prove this rigorously, we will heavily rely on semi-classical expansion techniques first developed
by Grenier \cite{Gre} and later extended by Carles in \cite{Car}.
The latter work in particular treats the case of a harmonic confinement and
nonlinearities of different strength (see also \cite{Car1} for an extension to
higher order nonlinearities and \cite{DLT} for results on the derivative NLS).
We can thus focus on the role played by the rotational forcing.
Indeed the motivation for our work is threefold: First, we aim to
generalize the results of \cite{Car} to the case where a rotational term is included
and consequently give a rigorous justification of \eqref{super}.
Second we strengthen these results a bit in the sense of Corollary \ref{th2.2} below. Finally
we aim to describe the dynamical features of rotational BECs from the
semi-classical point of view, as given in Theorem \ref{th2.3}.

\section{The modified WKB-approach}

As noticed in \cite{Car, Gre}, the classical Madelung
transformation \eqref{mad} is not well suited to rigorously derive the
semi-classical asymptotics of $\psi^\e$. Rather one is led
to consider a \emph{modified} version of it. To this end one
writes the \emph{exact solution} to \eqref{se} in the form
\begin{equation}\label{modWKB}
\psi^\epsilon (t,x) =a^\epsilon(t, x)\E^{\I \Phi^\e(t,x)/\e},
\end{equation}
where from now on the ``amplitude'' $a^\e$ is allowed to be
\emph{complex-valued}. Moreover $a^\e$ as well as the  (real-valued) phase
$\Phi^\e$ are assumed to admit an asymptotic expansion of the
form
\begin{equation} \label{exp}
a^\epsilon\sim a+ \e a_1 + \e^2
a_2+ \cdots, \quad \Phi^\epsilon\sim \Phi+ \e \Phi_1 + \e^2
\Phi_2+ \cdots.
\end{equation}
Since $a^\e$ is complex
valued, the phase $\Phi^\e$ can be seen as an \emph{additional}
degree of freedom introduced as a \emph{multiple scales
representation} for $\psi^\e$. In any case, the ansatz \eqref{modWKB} should not be
confused with \eqref{mad} and in particular it has nothing to do with a rewriting of
$\psi^\e$ into polar coordinates. The main gain of this
modified WKB-approach is that it yields a \emph{separation of
scales} within the appearing fast, \ie $\e$-oscillatory, phases and
slowly varying phases, which eventually can be included in the (complex-valued)
amplitudes.

An analysis based on WKB-type methods necessary requires admissible initial data.
To this end we introduce the following definition: A function
$f \in C^\infty(\R^d)$ is said to be \emph{sub-quadratic}, if
\begin{equation}\label{sq}
\forall \, \alpha \in \N^d, |\alpha| \geq 2: \sup_{x\in \R^d}|\partial^\alpha_x f| < \infty.
\end{equation}
We consequently impose:
\begin{Assumption} \label{ass1}
The initial data is of the form
\begin{equation}\label{initial}
\psi^\e_{\rm in} (x) =a^\epsilon_{\rm in}(x)\, \E^{\I \Phi_{\rm in}(x)/\e}.
\end{equation}
Here $a^\epsilon_{\rm in}$ is complex-valued and admits an asymptotic expansion in $\e$ whereas
$\Phi_{\rm in}(x)$ is $\e$-independent, real-valued, and sub-quadratic.
\end{Assumption}
It is important to note that Assumption \ref{ass1} imposes a particular $\e$-oscillatory
structure on the initial data but does \emph{not} yield
any problems at vacuum points since $a_{\rm in}^\e$ is allowed to be complex-valued.
\begin{Remark} \label{rem}
In particular we are free to choose, say in $d=2$ spatial dimensions, an initial data of the form:
$\Phi_{\rm in} (x) = 0$, $a^\e_{\rm in} (x) = \chi(r) \E^{ \I m \theta}$, where we have used polar
coordinates. Here $m \in \N_0$ is the so-called
\emph{winding number}. For $m \not = 0$ and $\chi$ appropriately chosen (see, \eg,
\cite{BDZ:2006, BMW, BKJZ}) this allows for so-called \emph{vortex states} as initial data.
We also note that we could allow for more general initial phase function
$\Phi_{\rm in}^\e$ which admit an asymptotic expansion in powers of $\e$ (analogously to $a_{\rm in}^\e$).
For simplicity we do nit treat this case though.
\end{Remark}
Upon substituting \eqref{modWKB} into the NLS (\ref{se}),
we have the freedom to split the appearing terms into
\begin{equation}\label{wkb}
\left \{
\begin{aligned}
& \partial_t\Phi^\e +\frac{1}{2}|\nabla_x \Phi^\e|^2 +V(x)-\Omega
x^\bot \cdot \nabla_x \Phi^\e
+f(|a^\e|^2)=0, \\
&\partial_t a^\e +\nabla_x \Phi^\e \cdot \nabla_x a^\e
+\frac{1}{2}a^\e \Delta \Phi^\e -\Omega x^\bot \cdot \nabla_x a^\e
=\frac{\I \epsilon}{2}\Delta a^\e.
\end{aligned}
\right.
\end{equation}
This system is \emph{equivalent} to the nonlinear
Schr\"odinger equation \eqref{se}. In particular it does no longer
represent a splitting into real and imaginary parts (since $a^\e$ is
complex-valued), in contrast Madelung's original approach. Moreover the
system \eqref{wkb} is seen to be perturbed by a term which is ``only'' $\O(\e)$ (since the
term $\propto \Delta a$ now appears in the equation for the amplitude instead of the
one for the the phase).

As $\e \to 0$, the corresponding \emph{limiting WKB system} is
then (formally) found to be
\begin{equation}\label{phia}
\left \{
\begin{aligned}
& \partial_t\Phi +\frac{1}{2}|\nabla_x  \Phi|^2 +V(x)-\Omega x^\bot
\cdot \nabla_x \Phi +f(|a|^2)=0, \\
&\partial_t a +\nabla_x\Phi \cdot \nabla_x a +\frac{1}{2}a \Delta \Phi
-\Omega x^\bot \cdot \nabla_x a =0.
\end{aligned}
\right.
\end{equation}
From here, setting $\rho = |a|^2$ and $\v = \nabla_x \Phi$, one (again
formally) obtains a hydrodynamical equations of superfluids \eqref{super}.
Indeed we shall prove below that \eqref{phia} and \eqref{super} are in a certain sense equivalent.

\begin{Remark} From the point of view of geometrical optics, the above given
limit corresponds to the \emph{supercritical case}, \cf \cite{Car, Car1, Car2} for more details.
Note however that a rotational forcing is neither considered in \cite{Car, Car1}, nor in \cite{GLM}.
Also note that in the case of a \emph{linear} Schr\"odinger equation the corresponding system \eqref{phia}
would be decoupled since the first equation would be replaced by the classical \emph{rotational
Hamilton-Jacobi equation} (HJ)
\begin{equation}\label{HJ}
    \partial_t S + \frac{1}{2}|\nabla_x S|^2 +V(x)-\Omega
    x^\bot\cdot \nabla_x S=0.
\end{equation}
Equations of the form \eqref{HJ} have been extensively studied in \cite{Liu:2006}, where several
qualitative properties for the corresponding solutions are established.
\end{Remark}

As a first step in our analysis we shall show that the usual hydrodynamical system \eqref{super}
and the WKB system \eqref{phia} admit smooth solutions on the same time-intervall.

\begin{Lemma}[Equivalence] \label{th2.1}
Let $T$ be the maximal time of existence for a smooth solution
$(\rho, \v)$, with $\rho\geq 0$, of the hydrodynamical system \eqref{super} and
let $T^{*}$ be the maximal existence-time of a smooth solution $(a, \Phi)$ of \eqref{phia}.
Then we have $T=T^{*}$.
\end{Lemma}
\begin{proof}
We define
$(\alpha, \beta, \v):=(\re a, \im a, \nabla_x
\Phi)$ and rewrite \eqref{phia} in the following form
\begin{equation}\label{wkb+}
\left \{
\begin{aligned}
& \partial_t \alpha +\v \cdot \nabla_x \alpha +\frac{\alpha}{2} \, \nabla_x \cdot \v -\Omega x^\bot \cdot \nabla_x \alpha=0,\\
& \partial_t \beta +\v \cdot \nabla_x \beta +\frac{\beta}{2} \, \nabla_x \cdot \v -\Omega x^\bot \cdot \nabla_x \beta=0,\\
& \partial_t \v + \nabla_x \left( \frac{|\v|^2}{2}  -\Omega
x^\bot \cdot \v + V + f(\alpha^2+\beta^2)\right)=0.
\end{aligned}
\right.
\end{equation}
Assume that  $(\alpha, \beta, \v)$ is a smooth solution of  the modified WKB
system (\ref{wkb+})  for $t\in [0, T^*)$. Then $\eta=\alpha^2+\beta^2$ and
$\v=\nabla_x \Phi$ are smooth solutions of (\ref{super}). By uniqueness it follows that $ (\eta,
\v)=(\rho, \v)$ for $t \leq T^*$ and hence $T^*\leq T$.

Conversely, assume that $(\rho, \v)$ is the smooth solution of the
hydrodynamic system (\ref{super}) for $t\in[0, T]$, subject to
initial data $(\rho_{\rm in} \geq 0, \v_{\rm in})$ such that $\rho_{\rm in}=\alpha_{\rm in}^2+\beta_{\rm in}^2$.
By assumption the velocity $\v$ is smooth, one thus obtains smooth
$a$ and $b$ from the first two transport equations in (\ref{wkb+}),
subject to initial conditions $(\alpha_{\rm in}, \beta_{\rm in})^\top$. A combination of
the two equations for $\alpha$ and $\beta$ gives
$$
\partial_t (\alpha^2+\beta^2)+\v \cdot \nabla_x (\alpha^2+\beta^2)=-(\alpha^2+\beta^2)\nabla_x \cdot
\v+\Omega x^\bot\cdot \nabla_x (\alpha^2+\beta^2).
$$
Subtracting this from the equation for $\rho$ in (\ref{super}), we
find that $\tilde \rho:=\rho-(\alpha^2+\beta^2)$ solves a transport equation
$$
\partial_t \tilde \rho +\nabla_x\cdot(\v \tilde \rho)=\Omega x^\bot\cdot
\nabla_x \tilde \rho \ ,
$$
with $\tilde \rho \big | _{t=0}=0$, and hence $\tilde
\rho (t,x)\equiv 0$, \ie $\rho=\alpha^2+\beta^2$. This shows that  $(\rho, \v)$ is
also the smooth solution of the hydrodynamic system (\ref{super})
for $t\in[0, T^*]$, hence $T\leq T^*$.

In summary this yields $T = T^*$ for
the solutions of \eqref{phia} and \eqref{wkb+}. To get back to $\Phi$ itself, we
first note that in the equation
$$
\partial_t\Phi +\frac{1}{2}|\nabla_x \Phi|^2 +V(x)-\Omega
x^\bot \cdot \nabla_x \Phi
+f(|a|^2)  = 0
$$
all terms are uniquely determined by \eqref{wkb+} except for $\partial_t \Phi$. Imposing
$\Phi \big | _{t=0} = \Phi_{\rm in}$, such that $\v_{\rm in} = \nabla_x \Phi_{\rm in}$, and setting
$$
\Phi (t,x) = \Phi_{\rm in} (x) + \int_0^t \left(\frac{1}{2} | \v(\tau,x)|^2 + |a(\tau,x)|^2 + V(x)-\Omega
x^\bot \cdot \v(\tau,x)\right) \D \tau ,
$$
we infer that $\partial_t(\nabla_x \Phi - \v) = \nabla_x \partial_t \Phi - \v = 0$, hence $\v = \nabla_x \Phi$.
This then fully determines $\Phi$ on $[0,T]$.
\end{proof}
Having established the equivalence between the limiting systems \eqref{super} and \eqref{phia}, we can now
focus on deriving rigorously \eqref{phia} from \eqref{wkb} (or, equivalently, the NLS \eqref{se}).

\section{The classical limit}

From now on we shall mainly consider the WKB system \eqref{phia} in the form (\ref{wkb+}) which
allows for a treatment in the sense of hyperbolic systems.
In the corresponding analysis, the rotational HJ equation (\ref{HJ}) becomes
important. As a preparatory step we shall therefore study the Cauchy problem
\begin{equation}\label{HJ1}
\left \{
\begin{aligned}
 & \partial_t S + \frac{1}{2}|\nabla_x S|^2 +V(x)-\Omega x^\bot\cdot \nabla_x S=  \ 0,\\
 & S \big|_{t=0} =  \ S_{\rm in}(x),
\end{aligned}
\right.
\end{equation}
The following result will be used throughout this work.´

\begin{Lemma} \label{lem} Let Assumption \ref{ass} hold. If $S_{\rm in}(x)\in C^\infty(\R^d) $ is sub-quadratic, then there
exists a $\tau >0$ %independent of $x \in \R^d$, 
such that (\ref{HJ}) admits a
unique smooth solution for $t\in[0, \tau)$. Moreover, the
phase $S(t,x)$ remains sub-quadratic in $x$, for all $t\in [0, \tau)$.
\end{Lemma}
As pointed out in \cite{Car}, the sub-quadratic assumption
for the initial phase is sharp for solving (\ref{HJ}) globally in
space even in the presence of no rotational force. Concerning the
existence of global-in-time smooth solutions we refer to \cite{Liu:2006}.
\begin{proof}
The Hamiltonian corresponding to \eqref{HJ1} is
$$
H=\frac{1}{2}|p|^2 +V(x)-\Omega x^\bot \cdot p.
$$
and thus the corresponding Hamiltonian flow is governed by
\begin{equation*}
\left \{
\begin{aligned}
\dot x& =  p-\Omega x^\bot, \quad x\big | _{t=0}=x_0, \\
 \dot p & = - \omega x-\Omega p^\bot, \quad p\big | _{t=0}=\nabla_x
 S_{\rm in}(x_0).
\end{aligned}
\right.
\end{equation*}
Standard ODE theory implies that there exists a unique solution
for $t\in [0, \tau]$, in which the map $x=x(t, x_0)$ is well defined, satisfying
$$
\det (\Gamma)>0, \quad \Gamma:=\frac{\partial }{\partial x_0} \, x(t,
x_0).
$$
Existence of smooth solutions is therefore guaranteed and we now
have to prove that $S$ remains sub-quadratic for $t \in [0, \tau]$. From \cite{Liu:2006} we infer that
the phase gradient $u=\nabla_x S$
satisfies
$$
\partial_t u +(u-\Omega x^\bot)\cdot \nabla_x u= - \omega  x -\Omega u^\bot.
$$
Further, along the particle path induced via $\dot x=u-\Omega x^\bot$ the Hessian of the phase function
$\Sigma:=D^2_x S(t,x)$ solves a matrix ODE
\begin{equation}\label{mk}
D_t \Sigma + \Sigma^2= - \omega {\bf I}{_d},
\end{equation}
where we shortly denote
$$
D_t:=\partial_t +(U-\Omega x^\perp)\cdot \nabla_x.
$$
This shows that $\Sigma$ remains uniformly bounded in terms of $x _0\in \R^d$.
Finally we show that the existence time $\tau$ does not shrink as
$x$ varies over $\R^d$.  Differentiation of $\dot x=u-\Omega
x^\bot$ in terms of $x_0$ yields
$$
D_t \Gamma=(\Sigma -\Omega J)\Gamma, \quad J=\left(%
\begin{array}{ccc}
  0 & 1 & 0 \\
  -1 & 0 & 0 \\
  0 & 0 & 0 \\
\end{array}%
\right),
$$
for $d=3$. (The expression for $J$ in the case $d=2$ is obvious.)
This gives
$$
\det(\Gamma)=\exp \left( \int_0^t \tr (\Sigma -\Omega J)\, \D \tau \right)=\exp
\left( \int_0^t \tr \Sigma \, \D \tau \right),
$$
where $\tr$ denotes the standard trace map. The fact that
$\det(\Gamma)$ does not depend on $x_0$ explicitly implies that $\Sigma$ is
uniformly bounded in $x_0 \in \R^d$.
\end{proof}

In order to prove the local-in-time existence for the system (\ref{wkb}),
we follow the strategy in \cite{Car} and decompose the phase $\Phi^\e$ into
\begin{equation}\label{decom}
\Phi^\e=\varphi^\e +S,
\end{equation}
where $S$ is the smooth, sub-quadratic phase function guaranteed by Lemma \ref{lem}.
In terms of $\varphi^\e$ and $ a^\e$ the system (\ref{wkb}) becomes
\begin{equation}\label{wkbneu}
\left \{
\begin{aligned}
& \partial_t\varphi^\e +\nabla_x S\cdot
\nabla_x \varphi^\e+\frac{1}{2}|\nabla_x \varphi^\e|^2  -\Omega x^\bot \cdot \nabla_x \varphi^\e
+f(|a^\e|^2)=0, \\
&\partial_t a^\e +\nabla_x(S+\varphi^\e)\cdot \nabla_x a^\e
+\frac{a^\e}{2} \, \Delta (S+\varphi^\e) -\Omega x^\bot \cdot \nabla_x
a^\e =\frac{\I \e}{2} \, \Delta_x a^\e.
\end{aligned}
\right.
\end{equation}
Note that this set of equations is still \emph{equivalent}
to the nonlinear Schr\"odinger equation \eqref{se}. The reason for
decomposing $\Phi^\e$ via \eqref{decom} is rather technical and due to the inclusion of the potential $V$ and
the rotational term.

For notational convenience we further introduce
$$
U^\e:=(\re a^\e,\,  \im a^\e, \, \partial_{x_1} \varphi^\e, \dots , \partial_{x_d} \varphi^\e)^\top,
$$
where $a^\e$ and $\varphi^\e$ satisfy \eqref{wkbneu} (an analogous notation is used for the corresponding initial data).
Moreover, we shall frequently use the notation
\begin{equation}\label{N}
\mathcal N[U^\e(t]):={\|\, U^\e(t) \, \|}_s + {\| \, |x|U^\e(t) \, \|}_{s-1},
\end{equation}
where ${\|\cdot\|}_s$ is the usual $H^s(\R^d)$-norm.

\begin{Proposition}[Local existence] \label{pro2.1} Denote by $\tau >0$ the existence time of smooth solution $S(t,x)$ to \eqref{HJ1}, and let the Assumption \ref{ass} hold. Consider
the Schr\"{o}dinger equation \eqref{se} subject to
initial data, which satisfy Assumption \ref{ass1} such that
$U^\e_{\rm in} \in H^s(\R^d)$ and $|x|U^\e_{\rm in} \in H^{s-1}(\R^d)$, for
$s>2+d/2$.
Then there exists a time $T_{\e}\in (0, \tau)$,
and a unique solution to \eqref{se}
of the following form
$$
\psi^\e(t,x)=a^\e(t, x)\E^{\I \Phi^\e(t,x)/\e}, \quad \text{for $0\leq t\leq T_{\e}$}.
$$
Moreover, it holds
$$ U^\e \in L^\infty((0, T_{\e}], H^s(\R^d)), \quad |x|U^\e \in
L^\infty((0, T_{\e}], H^{s-1}(\R^d)).$$
\end{Proposition}
We thus know that locally-in-time the oscillatory structure of the modified WKB representation for solutions to \eqref{se}
persists, as long as the classical rotational HJ equation has smooth solutions.
\begin{proof}
Introducing the velocities $v^\e:=\nabla_x \varphi^\e$ and $w:=\nabla_x S-\Omega x^\bot$ in (\ref{wkbneu}),
we have
\begin{equation}\label{3.7}
\left \{
\begin{aligned}
& \partial_t v^\e  +(v^\e+w) \cdot \nabla_x v^\e +(\nabla_x
w)v^\e + \nabla_x f(|a^\e|^2)=0, \\
&\partial_t a^\e +(v^\e +w)\cdot \nabla_x a^\e +\frac{a^\e}{2}
\nabla_x\cdot (w + v^\e) =\frac{\I \e}{2}\Delta_x a^\e.
\end{aligned}
\right.
\end{equation}
From Lemma \ref{lem} we know that $w=\nabla_x S-\Omega x^\perp$ is indeed
sub-linear. We further separate $a^\e$ into its real and imaginary part, \ie $a^\e
=\alpha^\e+\I \beta^\e$, to obtain the following hyperbolic system
\begin{equation}\label{hb}
    \partial_t U^\e +\sum_{j=1}^d
    \left(A_j(U^\e)+B_j(w)\right)\partial_{x_j}U^\e + {M}(\nabla_x
    w)U^\e=\frac{\epsilon}{2} \, LU^\e,
\end{equation}
where $U^\e=(\alpha^\e, \beta^\e, v_1^\e, \cdots, v_d^\e)^\top $. The coefficients matrices are
$$
\sum_{j=1}^d A_j\xi_j=\left(%
\begin{array}{ccc}
  v^\e\cdot \xi & 0 & \frac{\alpha^\e}{2}\xi^\top \\
  0 &   v^\e\cdot \xi & \frac{\beta^\e}{2}\xi^\top\\
  2f'\alpha^\e\xi &  2f'\beta^\e\xi & v^\e\cdot \xi \, {\bf I}_{d} \\
\end{array}%
\right), \quad \sum_{j=1}^d B_j\xi_j= w\cdot \xi \, {\bf I}_{d+2}
$$
with $f'=f'(|a^\e|^2+ |b^\e|^2)$ and
$$
{M}(\nabla_x w)=\left(%
\begin{array}{ccc}
  \frac{1}{2} \nabla_x \cdot w & 0  & 0 \\
   0 &  \frac{1}{2} \nabla_x \cdot w  & 0 \\
   0 & 0 & \nabla_x  w \\
\end{array}%
\right), \quad L=\left(%
\begin{array}{ccc}
  0 & -\Delta & 0  \\
  \Delta & 0 & 0  \\
  0 & 0 & {\bf 0}_{d} \\
\end{array}%
\right).
$$
Observe that \eqref{hb} can be symmetrized by
$$
Q=\left(%
\begin{array}{cc}
  {\bf I}_{2} & 0 \\
  0 & \frac{1}{4f'}{\bf I}_{d} \\
\end{array}%
\right),
$$
which explains the necessity for our assumption $f' >0$. (In \cite{Car1} it is shown how to
overcome this difficulty in case of higher order nonlinearities.)
We are thus able to proceed with energy estimates in the Sobolev
space $H^s(\R^d)$, which follow from the classical theory for hyperbolic systems:
\begin{Lemma} \label{lem1}
Assume that $w$ is sub-linear. For $s>2+d/2$, $|\alpha|\leq s$ and $|\beta|\leq s-1$,
there exists a locally bounded map $\mathcal{C}(\cdot)$,
satisfying $\mathcal{C}'>0$ and $\mathcal{C}(0)\geq 1$, such that
\begin{align}
\frac{\D}{\D t} \, \langle Q \, \partial_x^\alpha U^\e, \partial_x^\alpha U^\e \rangle & \leq
\mathcal{C}(N[U^\e]) \, {\|U^\e\|}_s^2 \, ,\\
\frac{\D}{\D t} \, \langle Q \, \partial_x^\beta (x_jU^\e), \partial_x^\beta (x_jU^\e)\rangle & \leq
\mathcal{C}(N[U^\e]) \,\left({\|U^\e\|}_s^2+{\|\, |x| U^\e\|}^2_{s-1}\right) \, ,
\end{align}
where $\langle \cdot, \cdot \rangle$ denotes the usual scalar product on $L^2(\R^d)$.
\end{Lemma}
Equipped with these estimates, we are able to conclude the local-in-time existence result. Set
$$
\mathcal{E}[U^\e]:=\sum_{|\alpha|\leq s }\langle Q\partial_x^\alpha U^\e,
\partial_x^\alpha U^\e\rangle +\sum_{j=1}^d\sum_{|\beta|\leq s-1 }\langle Q \, \partial_x^\beta
(x_jU^\e), \partial_x^\beta(x_j U^\e) \rangle,
$$
then the estimates presented above yield
$$
\mathcal{E}[U^\e(t)]\leq \mathcal{E}[U^\e(0)] + C \int_0^t
\mathcal{C}(\mathcal N[U^\e(\tau)]) \, \mathcal{E}[U^\e(\tau)]\, \D \tau .
$$
Invoking a Gronwall-type inequality we consequently arrive at
$$
\mathcal{E}[U^\e](t) \leq \mathcal{E}[U^\e(0)] \, \exp \left(C\int_0^t
\mathcal{C}(\mathcal N[U^\e](\tau)) \, \D \tau \right),
$$
which we expect to have a relaxed bound $2\mathcal{E}(0)$ for a
finite time where $\mathcal N[U^\e] \leq C_1$ for $C_1 > C_0 = \mathcal N[U^\e_{\rm
in}]$. Thus an existence time exists  and satisfies
$$
t \leq T_\e= \frac{\ln 2}{C \mathcal{C}(C_1)}.
$$
It is obvious that the smaller the initial data (measured by $C_1$),
the larger the time-interval of existence. A local-in-time existence is thus
established.
\end{proof}
\begin{Remark}The local-in-time existence for solutions of the limiting  system \eqref{wkb+} can be proved analogously.
\end{Remark}

For $\e \in (0, 1]$, assume that the initial data satisfies $\mathcal N[U_{\rm in}^\e]
\leq C_0<C$. Thus, for $\e \in (0, 1]$ fixed, the local existence shows that for any
number $C_1 \in (C_0, C)$, there exists a $T_\e>0$ so that
(\ref{phia}) has a unique classical solution satisfying $\mathcal N[U^\e]
\leq C_1$ for $ t\in [0, T_\e]$. Define
$$
T^\e:=\sup \{ 0<T_{\e} \leq \tau: \  \mathcal N[U^\e](t)\leq C_1, \
\forall \, t\in [0, T_\e]\}.
$$
Namely, $[0, T^\e)$ is the maximal time-interval of existence
and depends on $C_1$. It will be necessary to show that $\lim_{\e \to 0} T^\e>0$,
which we shall prove in Corollary \ref{th2.2}.

\begin{Theorem}[Convergence rates] \label{pro2.2}
Under the same assumptions as in Proposition \ref{pro2.1},
suppose that there exist $a_{\rm in}, \Phi_{\rm in} \in H^s(\R^d)$, for
$s>2+d/2$, such that
$$
{\|a_{\rm in}^\e-a_{\rm in}
\|}_s=\O(\epsilon) .
%\quad {\|\Phi^\e_{\rm in} -\Phi_{\rm in}\|}_s=\O(\e).
$$
Let $U:=(\re a,\,  \im a, \, \partial_{x_1} \varphi, \dots , \partial_{x_d} \varphi)^\top$
be the smooth solution to \eqref{phia} corresponding to
the initial
data $(\Phi_{\rm in}, a_{\rm in})$.  If
$$
U \in L^\infty((0, T^*], H^s(\R^d)), \quad |x|U \in L^\infty((0, T^*],
H^{s-1}(\R^d))
$$
with $T^*>0$ finite, then there exists $\e_0$ and $ C_* >0$, such
that for $\e \leq \e_0$
$$
{\| \, a^\e(t) -a(t) \, \|}_s \leq  C_* \e, \quad {\| \, \Phi^\e(t) -\Phi(t) \, \|}_s\leq C_* \e t,
$$
for all $t\in [0, \min\{T^*, T^\e \})$.
\end{Theorem}
As a consequence, we infer that the pair $(\rho, \v) = (|a|^2, \nabla_x \Phi)$, solves the
hydrodynamical system \eqref{super}. Note that $\Phi = \varphi + S$, where $S$ is determined by the
HJ equation \eqref{HJ}.
\begin{proof}
As $U^\e$ solves the forced hyperbolic system (\ref{hb}), then the
corresponding limiting function $U$ is governed by
\begin{equation}\label{hb+}
    \partial_t U +\sum_{j=1}^d
    (A_j(U)+B_j(w))\, \partial_{x_j}U +M(\nabla_x
    w)U =0.
\end{equation}
Denote $W^\e:= U^\e -U$, we thus have $LU^\e=LW^\e +LU$ and $W^\e$ solves
\begin{equation}\label{hbz}
    \partial_t W^\e +\sum_{j=1}^d
    (A_j(U^\e)+B_j(w))\partial_{x_j}W^\e +M(\nabla_x
    w)W^\e =\frac{\epsilon}{2}\, LW^\e +R^\e,
\end{equation}
where
$$
R^\e:=\sum_{j=1}^d \left(A_j(U)-A_j(U^\e) \right) \partial_{x_j} U +\frac{\e}{2}\, LU.
$$
Using Lemma \ref{lem1} we obtain, for
$s>2+d/2$,
\begin{equation}\label{R1}
\frac{\D}{\D t} \, \langle Q\partial_x^\alpha W^\e, \partial_x^\alpha W^\e \rangle  \leq
\mathcal{C}(N[W^\e])\|W^\e\|_s^2 +\langle Q\partial_x^\alpha R^\e, \partial_x^\alpha
W^\e \rangle, \quad |\alpha|\leq s,
\end{equation}
as well as
\begin{equation}\label{R2}
\begin{split}
\frac{\D}{\D t} \, \langle Q
\partial_x^\beta (x_jW^\e), \partial_x^\beta (x_jW^\e)\rangle  \leq & \
\mathcal{C}(N[W^\e])\left(\|W^\e\|_s^2+\|\, |x|W^\e\|^2_{s-1}\right) \\
& \ +\langle Q\partial_x^\beta (x_jR^\e), \partial_x^\beta (x_jW^\e)\rangle\;,
\end{split}
\end{equation}
for $|\beta|\leq s-1$.
Now we estimate terms involving $R^\e$  in (\ref{R1}) and
(\ref{R2}). This is done only for $t\in [0, \min\{T^*, T^\e\})$,
in which both $U^\e$ and $U$ are regular with uniform bounds for
$\mathcal N[U^\e]$ and $\mathcal N[U]$.

First we have
$$
\langle Q\partial_x^\alpha R^\e, \partial_x^\alpha W^\e \rangle \leq C {\|\partial_x^\alpha
W^\e\|}_0 \, {\|\partial_x^\alpha R^\e\|}_0
$$
and
$$
{\|\partial_x^\alpha R^\e\|}_0\leq \, C{\|\partial_{x_j} U\|}_s
{\|A_j(U^\e)-A_j(U)\|}_{|\alpha|}+C\epsilon {\|\partial_x^\alpha U\|}_0.
$$
Note that
$$
A_j(U^\e)-A_j(U)=(W^\e)_{j+2}{\bf I}_{d+2} + \tilde A,
$$
where the only non-zero entries of $\tilde A$ are $\tilde A_{1,
j+2}= \frac{1}{2}(W^\e)_1$, $\tilde A_{2, j+2}=\frac{1}{2}(W^\e)_2$ and
$$
\tilde A_{j+2, 1}=2f'(|a^\e|^2)a^\e-2f'(|a|^2)a, \quad \tilde
A_{j+2, 2}=2f'(|a^\e|^2)b^\e-2f'(|a|^2)b.
$$
Using these relations and the boundedness of $\|U\|_{s+1}$ we
conclude that
$$
{\|A_j(U^\e)-A_j(U)\|}_{|\alpha|} \leq C(\mathcal N[W^\e],
{\|U\|}_{s+1}) \, {\|W^\e\|}_s.
$$
Therefore
$$
\sum_{|\alpha|\leq s }\langle Q\partial_x^\alpha R^\e, \partial_x^\alpha W^\e \rangle \leq
C(\mathcal N[W^\e],{\|U\|}_{s+1}){\|W^\e\|}^2_s+ C\epsilon {\|U\|}_s {\|W^\e\|}_s.
$$
Further calculations give
\begin{equation*}
\begin{split}
\sum_{|\beta|\leq s-1 }\langle Q\partial_x^\beta(x_j R^\e), \partial_x^\beta(x_j W^\e)\rangle
\leq & \,  \mathcal{C}(\mathcal N[W^\e],{\|U\|}_{s}) \left(\|W^\e\|^2_s
+{\||x|W^\e\|}^2_{s-1}\right) \\
& \, + C\epsilon  {\|\, |x|W^\e\|}_{s-1}.
\end{split}
\end{equation*}
Substituting all these estimates into (\ref{R1}) and (\ref{R2})
and integrating over $\R^d$ we obtain
$$
\mathcal{E}[W^\e(t)]\leq \mathcal{E}[W^\e(0)]+C \int_0^t
\mathcal{C}(\mathcal N[W^\e], {\|U\|}_{s+1})\mathcal{E}[W^\e(\tau)] \, \D \tau+
C \e \int_0^t \mathcal N[W^\e(\tau)]\, \D \tau,
$$
where we have used the fact that $\mathcal{E}[W^\e]$ and
$\mathcal N^2[W^\e]$ are equivalent in the sense that there exists a
constant $C>0$ such that $C^{-1} \mathcal N^2[W^\e] \leq \mathcal{E}[W^\e]
\leq C \mathcal N^2[W^\e]$.
Note that $\mathcal{E}[W^\e(0)]=\O(\e^2)$, we thus have for $t\in
[0, T^* \wedge T]$
$$
\mathcal{E}[W^\e(t)]\leq C(T^*) \e^2 + c\int_0^t
\mathcal{C}(\mathcal N[W^\e], {\|U\|}_{s+1})\mathcal{E}[W^\e(\tau] \, \D \tau.
$$
We apply Gronwall's inequality to obtain
$$
\mathcal{E}[W^\e(t)] \leq C \e^2 \exp \left( c \int_0^t
\mathcal{C}(\mathcal N[W^\e], \|U\|_{s+1})\mathcal{E}[W^\e(\tau)] \, \D \tau
\right)=:\mathcal F(t).
$$
Thus, we infer
$$
\frac{\D}{\D t} \, \mathcal F(t) \leq C \mathcal{C}(\mathcal N[W^\e](t),
{\|U\|}_{s+1}(t)) \, \mathcal F(t)^2, \quad \mathcal F(0)=C\e^2.
$$
This differential inequality with $\O(\e^2)$ initial data ensures
that for fixed $T^*$, there exists $\e_0$  and $C^*$ such that for
$t \in [0, T^*]$ and $\e \leq \e_0$ we have
$$
\mathcal{E}[W^\e(t)]\leq \mathcal F(t) \leq (C^*\e)^2.
$$
This leads to
$$
{\|\, a^\e(t)-a(t) \,\|}_s=O(\e), \quad {\|\nabla_x \Phi^\e (t)-\nabla_x
\Phi (t)\|}_s=O(\e).
$$
The second estimate combined with the equations for $\Phi^\e$ and
$\Phi$ yields
$$
{\|\Phi^\e (t) -\Phi (t)\|}_s=O(\e)t
$$
and the proof is complete.
\end{proof}

Theorem \ref{pro2.2} yields an approximation of $\psi^\e $ \emph{for small times only} \cite{Car}, since
\begin{align*}
{\| \, \psi^\e (t) - a(t) \E^{\I \Phi(t) /\e} \, \|}_{L^2} =  & \, {\| \, a^\e (t)\E^{\I \Phi^\e(t) /\e} - a(t) \E^{\I \Phi(t) /\e} \|}_{L^2}\\
\leq & \,
{\| \, a^\e (t) - a(t) \, \|}_{L^2} + {\| \, \E^{\I \Phi(t) /\e} -  \E^{\I \Phi(t) /\e} \|}_{L^\infty} \, {\| \, a(t) \, \|}_{L^2} \\
\leq & \, \O(\e) + \O(t).
\end{align*}
In other words, to accurately approximate the wave function $\psi^\e$ itself one has to
take into account higher order \emph{corrections}. Indeed, it has been shown in \cite{Car} that
\begin{equation}\label{carles}
{\| \, \psi^\e (t) - a(t) \E^{ \I \Phi_1(t)} \E^{\I \Phi(t) /\e} \, \|}_{L^\infty((0,T^*];\, L^2(\R^d))}\leq \O(\e),
\end{equation}
where $\Phi_1$ is the first corrector appearing in the asymptotic expansion \eqref{exp}.
The slowly varying, phase $\Phi_1(t,x)$, for $t \in [0, T^*]$, is obtained from a linear hyperbolic system with source terms (and thus
will always be generated during the course of time). Consequently one might consider
${\tt a}:= a \E^{ \I \Phi_1}$ as a new (complex-valued) WKB amplitude.
Since in the present work we are mainly interested in deriving \eqref{super},
we shall not go into further details and rather refer to \cite{Car, Car1}. Note however, that in the
case where one aims to accurately describe the semi-classical dynamics of vortex states (see Remark \ref{rem})
it is \emph{crucial} to take into account this additional slowly varying phase $\Phi_1$. Even though the
fluid velocity in the classical limit is given by $\nabla_x \Phi$ as usually,
the total phase of the wave function in this
asymptotic regime is $\Phi + \e \Phi_1$. This $\O(\e)$-correction of the phase is usually ignored in the physics literature.
However, by doing so one can no longer justify a semi-classical approximation in the sense of \eqref{carles}.

With the above given result in hand, we are now able to show global-in-time convergence of the semi-classical limit.

\begin{Corollary} [Global convergence] \label{th2.2}
Under the same assumptions as before and for any $C_1$ satisfying
\begin{equation}\label{NN}
\mathcal N[U_0^\e]\leq C_0<C_1, \quad \mathcal N[U^\e(t)] \leq C_1 < C, \quad \mbox{for $t\in [0, \min\{T^*, T^\e \})$},
\end{equation}
it holds $T^\e(C_1)> T^*$ for $\e>0$ sufficiently small.
\end{Corollary}
In case the hydrodynamic system \eqref{super} is proved to admit global solution, \ie
$T^*=\infty$, we have consequently established convergence of
solutions of \eqref{wkb} towards solutions of \eqref{super} globally in time.

\begin{proof} Assume the contrary of what we aim to prove, \ie assume that there is a $C_1$ satisfying
(\ref{NN}) and a sequence $\e_n \to 0$ as $n\to \infty$ and
$T^{\e_n}(C_1)\leq T^*$. Then there exists $\tilde C $ satisfying
$$
\mathcal N[U(t)]< \tilde C <C_1.
$$
From Theorem \ref{pro2.2} it follows that
$$
\mathcal N[U^\e(t)-U(t)]\leq C_*\e.
$$
Thus, there is a $n\in \N$ such that $\mathcal N[U^{\e_n}(t)]\leq \tilde C$ for
$t\in[0, T^{\e_n})$.  On the other hand, we have
$$
\mathcal N[U^{\e_n}(t)] \leq \mathcal N[U^{\e_n}(t)-U(t)]+\mathcal N[U(t)] \leq C_*\e_0 +\tilde C, \quad
t\in [0, T^{\e_n}). $$ The uniform bound of $\mathcal N[U^{\e_n}(t)]$  enables
us to apply the local existence result again to extend the solution
beyond $T^{\e_n}$. This contradicts the definition of
$T^{\e_n}(C_1)$. Thus the proof is complete.
\end{proof}

\begin{Remark}
Corollary \ref{th2.2} is an extension of the corresponding theorem in
\cite{LT:2002}.
\end{Remark}

\section{Rotational dynamics of semi-classical superfluids}
The study of superfluid dynamics in response to a rotational
forcing has been the subject of vast experimental and
theoretical work in the past years. The expectation value of the
angular momentum, \ie
\begin{equation}\label{me}
    m^\e(t):=\I \epsilon \int_{\R^d} \overline{ \psi^\e}(t, x)\, x^\bot\cdot
    \nabla_x \psi^\e(t,x)\, \D x,
\end{equation}
has been mainly used to describe the dynamics
\cite{BDZ:2006}. In particular, a nonzero value of $m^\e(t)$
signifies the \emph{vortex nucleation} in BEC experiments.
For a condensate wave function of the form $\psi^\e (t,x) =a^\e (t,x)
\E^{\I \Phi^\e (t,x)/\e}$ we obtain
\begin{align*}
m^\e(t)& =-\int_{\R^d} |a^\e|^2x^\bot\cdot
    \nabla_x \Phi^\e(t,x)\D x + \I \epsilon \int_{\R^d} \bar a^\e x^\bot\cdot
    \nabla_x a^\e \, \D x\\
    & =-\int_{\R^d} |a|^2 x^\bot\cdot
    \nabla_x \Phi (t,x)\, \D x +\O(\e).
\end{align*}
In the following we denote the leading order angular momentum by
\begin{equation}
m(t):=-\int_{\R^d} \rho(t,x) \, x^\bot\cdot
    \v (t,x)\, \D x \, ,
\end{equation}
where, as before, $(\rho, \v)=(|a|^2,
\nabla_x \Phi)$. We shall also use the notation
$$ {\langle g \rangle}_{\rho (t)}
:=\int_{\R^d} g(x) \rho (t,x) \, \D x,
$$
for any smooth function $g(x)$.

\begin{Corollary}\label{th2.3}
Let $f(z) =z$ and impose the same assumptions as before. Then, as $\e \to 0$ it holds
\begin{equation}\label{am}
m^\e(t)=m(0)+\frac{\Omega}{2}\left( \langle |x|^2\rangle_{\rho(t)}-
\langle |x|^2\rangle_{\rho_{\rm in}}\right)+ \O(\e).
\end{equation}
Moreover we have
\begin{equation}\label{amp}
\frac{\rm d}{{\rm d} t} m^\e(t)=\Omega \langle x\cdot \v \rangle_{\rho (t)} +
\frac{\delta}{2\omega^2_{\bot}} \langle x_1x_2 \rangle_{\rho (t)}  +\O(\e),
\end{equation}
where $\delta
=\frac{\omega_{1}^2-\omega_{2}^2}{\omega_{1}^2+\omega_{2}^2}$ denotes the trap deformation and
$\omega^2_{\bot}=\frac{1}{2}(\omega_{1}^2+\omega_{2}^2)$ the radial frequency.
\end{Corollary}
In the semi-classical regime, the angular momentum expectation value
is dominated by the classical rotational effect, in contrast to the results of \cite{BDZ:2006, BKJZ}
(where the GPE is considered unscaled). The result also shows that a non-isotropic frequency for
the trapping potential contributes to the change of the angular momentum.

\begin{proof}
Let
$$
n(t):=\int_{\R^d} \rho(t,x) x \cdot \v(t,x) \, \D x
$$
and we also introduce
$$
X^\e(t):=\int_{\R^d} |x|^2 |\psi^\e(t,x)|^2\D x = {\langle |x|^2 \rangle}_{\rho(t)}+ \O(\e),
$$
A straightforward calculation calculation then yields
\begin{equation}\label{dm}
\frac{\D }{\D t} \, m(t) = \Omega n(t) +\int_{\R^d} \rho (x^\bot\cdot
\nabla_x )V \, \D x.
\end{equation}
When combined with a quadratic potential of the form
$V=\frac{1}{2}|\omega \cdot x|^2$, this leads to the relation (\ref{amp}).
Also, one easily verifies that
$$
\frac{\D }{\D t} \, {\langle |x|^2 \rangle}_{\rho(t)} =2 n(t).
$$
which, upon inserting into (\ref{dm}), yields
(\ref{am}) after an integration w.r.t. time.
\end{proof}
We also remark that the change of $n(t)$ also depends on $m(t)$ as
well as the total energy. This can be seen as follows: First we calculate
$$
\frac{\D }{\D t} \, n(t)=-\int_{\R^d} \rho(x\cdot \nabla_x )(V(x)
+f(\rho)) \, \D x+ \int_{\R^d}\rho\v^2\, \D x
$$
and an integration by parts gives
$$
\frac{\D }{\D t} \, n(t)=\int_{\R^d}(  \rho^2 +\rho \v^2
) \, \D x -\int_{\R^d} \rho (x\cdot \nabla_x )V(x) \, \D x.
$$
Note that for $f(z)=z$, the total energy (\ref{energy}) expressed in terms of $a$ and $\v$ reads
$$
E=\int_{\R^d} \frac{1}{2}\rho |\v -\Omega
x^\bot |^2  +\rho \left(V-\frac{\Omega^2}{2}(x_1^2+x_2^2)\right)
+\frac{1}{2}\rho^2 \, \D x +\O(\e).
$$
For quadratic potentials $V=\frac{1}{2}|\omega \cdot x|^2$, it holds
$(x\cdot \nabla_x )V(x)=2V(x)$, which consequently implies
\begin{align*}
\frac{\D }{\D t} \, n(t)& =2\left( E-\Omega m(t) -\int_{\R^d} \rho V(x) \, \D x
\right)
-\int_{\R^d} \rho (x\cdot \nabla_x )V(x)\, \D x\\
& =2(E-\Omega m(t)) -4 \int_{\R^d} \rho V(x)\, \D x.
\end{align*}
Combining these calculations yields the following closed system
\begin{equation*}
\left \{
\begin{aligned}
\frac{\D }{\D t} \, n(t) = & \, 2\left(E(0)-\Omega m(t) \right)- 2
\int_{\R^d} |\omega \cdot x|^2 \rho \, \D x,\\
\frac{\D }{\D t}\, m(t)=& \, \Omega n(t)+(\omega^2_{x_1}
-\omega^2_{x_2})\int_{\R^d}x_1x_2\rho \, \D x.
\end{aligned}
\right.
\end{equation*}
Finally, we note that for an \emph{isotropic} harmonic confinement, \ie $V=\frac{1}{2}\omega^2|x|^2$, we have
$$
\frac{\D }{\D t} \, m(t) -\Omega n(t)=0, \quad \frac{\D }{\D t} \, n(t) +
2\Omega m(t)=2E(0)-2\omega^2 {\langle |x|^2 \rangle}_{\rho(t)}.
$$
Since in this case, it also holds
$$
\frac{\D }{\D t} \, X=2n(t)=\frac{2}{\Omega}\frac{\D }{\D t} \, m(t)
$$
and we simply obtain
$$
m(t)=\frac{\Omega}{2}X(t) +m(0)-\frac{\Omega}{2}{\langle |x|^2 \rangle}_{\rho_{\rm in}}.
$$
The obtained relations yield a closed equation for $m(t)$, \ie
$$
\frac{\D^2 }{\D t^2} \, m(t) +(2\Omega^2+4\omega^2)m(t)=2\Omega E(0)+4m(0)\omega^2
-2\Omega \omega^2{\langle |x|^2 \rangle}_{\rho_{\rm in}}.
$$
The general solution of it can be written as
$$
m=C_1 \cos \sqrt{4\omega^2
+2\Omega^2}t+C_2 \sin\sqrt{4\omega^2+2\Omega^2}t + \frac{\Omega
E(0)+ \Omega \omega^2 (2m(0)-{\langle |x|^2 \rangle}_{\rho_{\rm in}}) }{\Omega^2 +2\omega^2},
$$
where $C_1$ and $C_2$ are determined by $m(0)$ and $\frac{\rm d}{{\rm d} t}
m(t)\Big |_{t=0}=\Omega n(0)$. The above given calculations could be used to compare
numerical simulations of the full NLS dynamics in the spirit of \cite{BDZ:2006, BKJZ}.

\bigskip

{\bf Acknowledgment.} The authors thank R. Carles for helpful discussions. 
H. Liu wants to thank the WPI (Vienna) for its hospitality and support during his visit in May 2006 
when this work was initiated.

\end{document}